\documentclass[reqno,a4paper,12pt]{amsart}

\usepackage[pdfpagelabels]{hyperref}
\usepackage{amssymb}
\usepackage{url}

\newtheorem{theorem}{Theorem}
\newtheorem{lemma}[theorem]{Lemma}
\newtheorem{corollary}[theorem]{Corollary}

\theoremstyle{definition}

\theoremstyle{remark}
\newtheorem{remark}{Remark}
\newtheorem{example}{Example}

\newcommand{\nm}[1]{\lVert#1\rVert}

\newcommand{\D}{\mathbb{D}}
\newcommand{\B}{\mathcal{B}}

\newcommand{\N}{\mathbb{N}}
\newcommand{\Z}{\mathbb{Z}}

\newcommand{\C}{\mathbb{C}}

\renewcommand{\phi}{\varphi}

\newcommand{\BMOA}{\rm BMOA}
\newcommand{\VMOA}{\rm VMOA}

\begin{document}

\title[Slowly growing solutions of ODEs revisited]{Slowly growing solutions of ODEs revisited}
\thanks{The author is supported  in part by the Academy of Finland \#286877.}

\author{Janne Gr\"ohn}
\address{Department of Physics and Mathematics, University of Eastern Finland\\ 
\indent P.O. Box 111, FI-80101 Joensuu, Finland}
\email{janne.grohn@uef.fi}

\date{\today}


\begin{abstract}
Solutions of the differential equation $f''+Af=0$ are considered assuming
that $A$ is analytic in the unit disc $\D$ and satisfies
\begin{equation} \label{eq:dag} 
\sup_{z\in\D} \, |A(z)| (1-|z|^2)^2 \log\frac{e}{1-|z|} < \infty. \tag{$\star$}
\end{equation}
By recent results in the literature, such restriction has been associated to coefficient 
conditions which place all solutions in the Bloch space~$\mathcal{B}$.
In this paper it is shown that any coefficient condition implying \eqref{eq:dag} fails
to detect certain cases when Bloch solutions do appear.
The converse problem is also addressed: What can be said about the growth of the coefficient $A$
if all solutions of $f''+Af=0$ belong to $\mathcal{B}$?
An overall revised look into slowly growing solutions
is presented, emphasizing function spaces $\mathcal{B}$, $\BMOA$ and $\VMOA$.
\end{abstract}

\maketitle


\section{Introduction}
 
Let $\mathcal{H}(\D)$ denote the collection of analytic functions in the (open) unit disc~$\D$
of the complex plane $\C$. It is well-known that the growth of the coefficient $A\in\mathcal{H}(\D)$ controls the 
growth of solutions $f\in\mathcal{H}(\D)$ of the linear differential equation
\begin{equation} \label{eq:de2}
  f''+Af=0,
\end{equation}
and vice versa. The recent study \cite{GHR:preprint} concerns conditions, given in terms of the coefficient~$A$, which
imply that all solutions of \eqref{eq:de2} belong to a~given space of slowly growing analytic functions.
Special attention is paid to $\B$ (Bloch space), $\BMOA$ (analytic functions of bounded mean oscillation) 
and $\VMOA$ (analytic functions of vanishing mean oscillation). These coefficient conditions
have in common that they all imply 
\begin{equation*}
  \nm{A}_{\mathcal{L}^1} = \sup_{z\in\D} \, |A(z)| (1-|z|^2)^2\log\frac{e}{1-|z|} < \infty,
\end{equation*}
which is the subject of this research.
The operator theoretic approach in \cite{GHR:preprint} is based on duality relations,
in contrast to this paper, where more classical tools are employed.

The search for coefficient conditions forcing all solutions of \eqref{eq:de2} to be of slow growth
has been active for many years. 
In the 1997 summer school \emph{Function Spaces and Complex Analysis} (Mekrij\"arvi Research Station, Finland), 
N.~Danikas posed the following problem:
\begin{enumerate}
\item[\rm (Q)] Find a~sharp condition for the coefficient $A$ which implies that all solutions of~\eqref{eq:de2}
belong to $\mathcal{B}$.
\end{enumerate}
It is known that, if $\nm{A}_{\mathcal{L}^1}$ is sufficiently small, then all solutions of \eqref{eq:de2} belong
to $\mathcal{B}$. This result was recently discovered with the best possible upper bound
for $\nm{A}_{\mathcal{L}^1}$ in \cite[Corollary~4(b) and Example~5(b)]{HKR:2016}.
This means that in the language of $\mathcal{L}^1$-norms, the problem (Q) has been solved.
The alternative approach in \cite{GHR:preprint} produces a family of coefficient
conditions, which all fall into the category $A\in\mathcal{L}^1$, see \cite[Theorems~10 and 11]{GHR:preprint}.

Our intention is to take a revised look into slowly growing solutions of \eqref{eq:de2}, and in particular,
to concentrate to the borderline case $A\in\mathcal{L}^1$. We show that any coefficient
condition implying $A\in\mathcal{L}^1$ is not
sufficiently delicate to detect certain special cases when Bloch solutions do appear. In this sense, the problem~(Q) remains open
as the most natural description is yet to be found. The converse problem is addressed in
Section~\ref{sec:converse}.


\section{Results}


\subsection{Growth of solutions}

Our first result solves the problem (Q) 
in terms of the maximum modulus $M_\infty(r,A) = \max_{|z|=r} |A(z)|$, $0\leq r<1$.


\begin{theorem} \label{thm:imp}
Let $A\in\mathcal{H}(\D)$. If there exists $0\leq r_0<1$ such that
\begin{equation} \label{eq:impass}
  \sup_{r_0<r<1} \, M_\infty(r,A) (1-r)^2  \exp \!\left( \, \int_{r_0}^r M_\infty(t,A) (1-t) \, dt \right) < \infty,
\end{equation}
then all solutions of \eqref{eq:de2} belong to $\mathcal{B}$.
\end{theorem}

Theorem~\ref{thm:imp} is based on a~representation formula for solutions of \eqref{eq:de2} 
and the following elementary observation.
If $f$ is a~solution of \eqref{eq:de2} for $A\in\mathcal{H}(\D)$, then $f$ belongs to the Bloch space
\begin{equation*}
  \mathcal{B} = \Big\{ f\in\mathcal{H}(\D) : \nm{f}_{\mathcal{B}} = \sup_{z\in\D} |f'(z)| (1-|z|^2) < \infty \Big\}
\end{equation*}
if and only if 
\begin{equation} \label{eq:bb}
  \sup_{z\in\D} |f(z)| |A(z)| (1-|z|^2)^2< \infty.
\end{equation}
Theorem~\ref{thm:imp} sharpens \cite[Corollary~4(b)]{HKR:2016},
but fails to be an optimal solution to the problem (Q) as it shares the same defects with other known solutions;
see Remarks~\ref{remark:defects} and~\ref{remark:imp} below.

The growth space $\mathcal{L}^\alpha$ for $0\leq \alpha<\infty$ consists of those $A\in\mathcal{H}(\D)$ for which
\begin{equation*}
  \nm{A}_{\mathcal{L}^\alpha} = \sup_{z\in\D} \, |A(z)| (1-|z|^2)^2 \left( \log\frac{e}{1-|z|} \right)^\alpha < \infty.
\end{equation*}
The space $\mathcal{L}^0$ appears several times in the literature, and is usually denoted
by $H^\infty_2$ or $\mathcal{A}^{-2}$. In the sense of \eqref{eq:bb} it seems to be the correct
ballpark for the study of Bloch solutions of \eqref{eq:de2}. However, 
even if $\nm{A}_{\mathcal{L}^0}$ is arbitrarily small, it is possible that
all non-trivial solutions ($f\not\equiv 0$) of \eqref{eq:de2} lie outside~$\mathcal{B}$; see Example~\ref{ex:nonbloch} below.
If $A\in\mathcal{L}^\alpha$ for $1<\alpha<\infty$, then all solutions of \eqref{eq:de2} are bounded
in $\D$ by \cite[Theorem~4.2]{H:2000}. As explained in the Introduction, if $\nm{A}_{\mathcal{L}^1}$ is sufficiently
small, then all solutions of \eqref{eq:de2} belong to the Bloch space,
while the weaker condition $A\in\mathcal{L}^1$ allows some solutions to lie outside~$\mathcal{B}$.
The following result 
is in line with the heuristic principle which claims that
\emph{small change in $\nm{A}_{\mathcal{L}^1}$ has a~huge impact on solutions of \eqref{eq:de2}}.


\begin{theorem} \label{thm:nn}
If $\nm{A}_{\mathcal{L}^1} < 4/n$ for $n\in\N$, then all solutions $f$ of \eqref{eq:de2}
satisfy $f,f^2, \dotsc,f^n \in \mathcal{B}$.
\end{theorem}

For $1/2<\alpha<\infty$, the coefficient condition $A\in\mathcal{L}^\alpha$
places all solutions of \eqref{eq:de2} 
in $\bigcap_{0<p<\infty} H^p$, see \cite[Corollary~1.9]{R:2007}. 
This property is no longer true for $\alpha=1/2$ as
certain solutions may lie outside the Nevanlinna class $\mathcal{N}$; 
apply \cite[Theorem~4]{P:1982} to $Q(r)=(1-r)^{-2} ( \log(e/(1-r)))^{-1/2}$, $0\leq r<1$.
It seems that non-Nevanlinna solutions produced in this manner
do not belong to $\mathcal{B}$ as they are exponentials of very badly behaved Bloch functions
themselves. 
The following result indicates that not all Bloch solutions of \eqref{eq:de2}
are smooth enough to be contained in $\mathcal{N}$. By the discussion above, these solutions cannot
be detected by any coefficient condition which implies $A\in \mathcal{L}^1$.

As usual, the Hardy space $H^p$ for $0<p<\infty$ consists of
$f\in \mathcal{H}(\D)$ for which
\begin{equation*}
  \nm{f}_{H^p}^p = \lim_{r\to 1^-} \, \frac{1}{2\pi} \, \int_0^{2\pi} |f(re^{i\theta})|^p \, d\theta < \infty,
\end{equation*}
while the Nevanlinna class $\mathcal{N}$ contains $f\in\mathcal{H}(\D)$ such that
\begin{equation*}
  \lim_{r\to 1^-} \, \frac{1}{2\pi} \, \int_0^{2\pi} \log^+ |f(re^{i\theta})| \, d\theta<\infty,
  \quad \log^+ = \max \{ \log, 0\}.
\end{equation*}


\begin{theorem} \label{thm:ex}
Let  $0<C<\infty$. Then, there exists a coefficient $A\in\mathcal{H}(\D)$ with $\nm{A}_{\mathcal{L}^0} < C$ 
such that \eqref{eq:de2} admits a (zero-free) solution $f\in \mathcal{B} \setminus \mathcal{N}$.
\end{theorem}

The following result complements Theorem~\ref{thm:ex}
by offering a condition under which non-Nevanlinna solutions do not appear.


\begin{theorem} \label{thm:zfbase}
If $\nm{A}_{\mathcal{L}^0}\leq 1$ and there exists one zero-free solution of \eqref{eq:de2} 
which belongs to $\bigcup_{0<p<\infty} H^p$, then all solutions
of \eqref{eq:de2} are in $\bigcup_{0<p<\infty} H^p$.
\end{theorem}

The coefficient condition $\nm{A}_{\mathcal{L}^0} \leq 1$ corresponds to the classical
univalency criterion \cite[Theorem~I]{N:1949} due to Nehari, which
implies that all non-trivial solutions of \eqref{eq:de2} have at most one zero in $\D$.
Theorem~\ref{thm:zfbase} should be compared to \cite[Theorem~4]{H:2013}
which holds in a~more general setting.


\subsection{Oscillation of solutions}

If $A\in\mathcal{H}(\D)$ and there exists $0<R<1$ such that
$|A(z)| (1-|z|^2)^2 \leq 1$ for all $R<|z|<1$,
then all non-trivial solutions of \eqref{eq:de2} vanish at most finitely many times in $\D$ \cite[Theorem~1]{S:1955}.
This is the case, in particular, if $A\in\mathcal{L}^\alpha$ for any $0<\alpha<\infty$.
The following example concerns a~case when all solutions belong to $\mathcal{B}$
while one of them has infinitely many zeros. This is Hille's example, see \cite{H:1949} and \cite[p.~162]{S:1955}.


\begin{example} \label{ex:hille}
Let $0<\gamma<\infty$. On one hand, all solutions of the differential equation~\eqref{eq:de2} for
$A(z)=(1+4\gamma^2)/(1-z^2)^2$, $z\in\D$,
are bounded and hence in $\mathcal{B}$. This follows from the estimates in \cite[p.~131]{S:2012}, for example. 
On the other hand, the particular solution
\begin{equation*}
  f(z)  = \sqrt{1-z^2} \, \sin \!\left( \gamma \log\frac{1+z}{1-z} \right), \quad z\in\D,
\end{equation*}
has  infinitely many (real) zeros $z_n = (e^{\pi n/\gamma}-1)/(e^{\pi n/\gamma}+1)$, $n\in\Z$.
\hfill $\diamond$
\end{example}


\begin{remark} \label{remark:defects}
By the discussion above, the coefficient condition $A\in\mathcal{L}^1$
implies that all non-trivial solutions of \eqref{eq:de2} belong to $\bigcap_{0<p<\infty} H^p$ 
and have at most finitely many zeros. We have shown that neither of these properties is 
characteristic to Bloch solutions of \eqref{eq:de2} under the restriction $A\in \mathcal{L}^0$.
\end{remark}

We point out that, although $A\in\mathcal{L}^1$ is not sufficient to
place all solutions of \eqref{eq:de2} in $\mathcal{B}$, it
guarantees that solutions are normal in the sense
\begin{equation*}
\sup_{z\in\D} \, f^{\#}(z) (1-|z|^2) = \sup_{z\in\D} \, \frac{|f'(z)|}{1+|f(z)|^2} \, (1-|z|^2) < \infty.
\end{equation*}
This follows from \cite[Proposition~7]{GNR:preprint} by using the fact that
all non-trivial solutions have at most finitely many zeros provided that $A\in\mathcal{L}^1$.


\subsection{Solutions of finite valance}
Let $n(f,\zeta) = \# \{ z\in\D : f(z)=\zeta\}$
be the counting function for $\zeta$-points of $f\in\mathcal{H}(\D)$;
let $D(z,r)$ denote the Euclidean disc of radius $0<r<\infty$ centered at $z\in\D$;
and let $dm$ be the Lebesgue area measure.
According to \cite[Satz~1]{P:1977}, if $f\in\mathcal{B}$ and
\begin{equation} \label{eq:intnumber}
  V_f= \sup_{z\in\C} \, \int_{D(z,1)} n(f,\zeta)\, dm(\zeta) < \infty,
\end{equation}
then $f\in\BMOA$. Hence, Bloch functions of finite valence belong to~$\BMOA$.
Recall that $f\in\BMOA$ if and only if
$\nm{f}_{\BMOA}^2 = \sup_{a\in\D}\, \nm{g_a}_{H^2}^2<\infty$,
where $g_a(z) = f(\varphi_a(z)) - f(a)$
and $\varphi_a(z)=(a-z)/(1-\overline{a}z)$ for $a,z\in\D$.

If $\nm{A}_{\mathcal{L}^1}$ is sufficiently small, then all 
finitely valent solutions of \eqref{eq:de2} are not only in $\BMOA$
but also possess a~specific type of regularity.


\begin{theorem} \label{thm:bmoa_not}
Let $A\in \mathcal{L}^1$. If $f$ is a~solution of \eqref{eq:de2}
which satisfies \eqref{eq:intnumber}, then 
\begin{equation} \label{eq:newest}
  \int_{\D} |f'(z)|^2 \left( \log\frac{e}{1-|z|} \right)^{-\beta} dm(z)<\infty
\end{equation}
for any $\nm{A}_{\mathcal{L}^1}/2<\beta<\infty$.
\end{theorem}

Example~\ref{ex:valent} below shows that, regardless of the size of $\nm{A}_{\mathcal{L}^1}$,
both finitely and infinitely valent (non-trivial) solutions of \eqref{eq:de2} are possible.


\subsection{Converse problem} \label{sec:converse}

Before going any further, 
we discuss a~problem converse to Theorem~\ref{thm:imp}: How is the growth of the coefficient
$A\in\mathcal{H}(\D)$ restricted if all solutions of \eqref{eq:de2} are in $\mathcal{B}$?

The argument in \cite{S:2012} reveals the following estimates.
Let $f_1,f_2$ be linearly independent \emph{bounded} solutions of \eqref{eq:de2} for $A\in\mathcal{H}(\D)$.
Without any loss of generality, we may assume that
$f_1 f_2' - f_1'f_2 = 1$. By a straight-forward computation
$A = f_1'f_2''- f_1''f_2'$, and therefore $\sup_{z\in\D} |A(z)| (1-|z|^2)^3 < \infty$.
Moreover, the spherical derivative $w^{\#} = |w'|/(1+|w|^2)$ of $w=f_1/f_2$ satisfies
$w^{\#} = 1/(|f_1|^2 + |f_2|^2) \leq |f_1'|^2 + |f_2'|^2$, and hence 
$\sup_{z\in\D} \, w^{\#}(z) (1-|z|^2)^2 < \infty$.
It is clear that these estimates withstand the weaker assumption $f_1,f_2\in\mathcal{B}$. 
The following result improves the growth estimate for $A$ 
and is related to a~problem mentioned in \cite[p.~131]{S:2012}.


\begin{theorem} \label{thm:conv_preli}
Let $f_1,f_2\in\mathcal{B}$ be linearly independent solutions of \eqref{eq:de2} for $A\in\mathcal{H}(\D)$.
Then, $\sup_{z\in\D} |A(z)|(1-|z|^2)^{5/2}\lesssim\max\{\nm{f_1}_{\mathcal{B}},\nm{f_2}_{\mathcal{B}}\}<\infty$.
\end{theorem}

Here $\lesssim$ denotes a one sided estimate up to a constant.
The betting is that Theorem~\ref{thm:conv_preli} is not sharp.
It would be desirable to show $A\in\mathcal{L}^0$ if $f_1,f_2\in\mathcal{B}$.
We do not know whether this is true (even for $f_1,f_2$ bounded), however.
Theorem~\ref{thm:conv_preli} fails to be true if we have information only on one non-trivial
solution of \eqref{eq:de2}. For example, $f(z)=\exp(-(1+z)/(1-z))$ is a bounded solution of \eqref{eq:de2}
for $A(z)=-4z/(1-z)^4$, $z\in\D$. In this case \eqref{eq:de2} admits also non-Bloch solutions such as
\begin{equation*} 
  f(z) \int_0^z \frac{1}{f(\zeta)^2} \, d\zeta, \quad z\in\D,
\end{equation*}
which is linearly independent to $f$ and grows too fast on the positive real axis to be included in $\mathcal{B}$
(by the Bernoulli-l'H\^{o}pital theorem).

Let $A\in\mathcal{H}(\D)$. If there exist linearly independent \emph{bounded} solutions
$f_1, f_2$ of \eqref{eq:de2} such that $\inf_{z\in\D} \, (|f_1(z)|+ |f_2(z)|) >0$, then $A\in \mathcal{L}^0$
by an~argument based on the corona theorem \cite[p.~3]{GHR:preprint}.
We extend this observation for $\mathcal{B}$ with an~argument independent of the corona theorem.


\begin{theorem} \label{thm:conv}
Let  $f_1,f_2\in\mathcal{B}$ be linearly independent solutions of \eqref{eq:de2} for $A\in\mathcal{H}(\D)$
such that $\inf_{z\in\D} \, (|f_1(z)|+ |f_2(z)|) >0$. Then, $A\in \mathcal{L}^0$ and
$(f_1/f_2)^{\#}$ is bounded in $\D$.
\end{theorem}


\subsection{Solutions of bounded and vanishing mean oscillation} \label{sec:25}

Coefficient conditions, which place all solutions of \eqref{eq:de2} in $\BMOA$, are considered
in \cite{GHR:preprint}. We derive a result similar to \cite[Theorem~3]{GHR:preprint} by using known growth estimates
for solutions of \eqref{eq:de2}. This method is somewhat surprising, since
it was not known to work with slowly growing solutions.
By the Carleson measure description in
\cite[Theorem~1]{Z:2003}, Theorem~\ref{thm:bmoa} is weaker than \cite[Theorem~3]{GHR:preprint}.


\begin{theorem} \label{thm:bmoa}
Let $A\in\mathcal{H}(\D)$. If 
\begin{equation} \label{eq:bmoaa}
  \sup_{a\in\D} \, \left( \log\frac{e}{1-|a|} \right) \int_{\D} |A(z)| (1-|\varphi_a(z)|^2) \, dm(z)
\end{equation}
is sufficiently small, then all solutions of \eqref{eq:de2} belong to $\BMOA$.
\end{theorem}

Coefficient conditions, which place all solutions of \eqref{eq:de2} in $\VMOA$, are also discussed
in \cite{GHR:preprint}. We consider two related results which, as opposed to ones in \cite{GHR:preprint}, 
are given in terms of the radial growth of the coefficient.
Recall that $f\in\VMOA$ if and only if $\nm{g_a}_{H^2}^2\to 0^+$ as $|a|\to 1^-$.


\begin{theorem} \label{thm:vmoa}
Let $A\in\mathcal{H}(\D)$. If there exists $0\leq r_0<1$ such that
\begin{equation} \label{eq:vmoageneral}
  \int_{r_0}^1 M_\infty(r,A)^2 
  \exp\!\left(\, 2 \int_{r_0}^r M_\infty(t,A) (1-t) \, dt \right) (1-r^2)^3\, dr<\infty,
\end{equation}
then all solutions of \eqref{eq:de2} belong to $\VMOA$.
\end{theorem}

Theorem~\ref{thm:vmoa} gives rise to the following corollary.
The coefficient condition \eqref{eq:corvmoa}
allows solutions of \eqref{eq:de2} to be unbounded, see Example~\ref{ex:unbounded} below.


\begin{corollary}\label{cor:vmoa}
Let $A\in\mathcal{H}(\D)$. If
\begin{equation} \label{eq:corvmoa}
 \sup_{z\in\D} \, |A(z)| (1-|z|^2)^2 \left( \log\frac{e}{1-|z|} \right) \log \log\frac{e}{1-|z|} < \infty,
\end{equation}
then all solutions of \eqref{eq:de2} belong to $\VMOA$.
\end{corollary}


\section{Proof of Theorem~\ref{thm:imp}} \label{sec:f}

The following proof is based on
the growth estimate \cite[Theorem~4.2]{H:2000} for solutions of \eqref{eq:de2}.
The known approaches to Bloch solutions of \eqref{eq:de2}
depend on other methods (duality relations \cite{GHR:preprint} and straight-forward integration \cite{HKR:2016}).


\begin{proof}[Proof of Theorem~\ref{thm:imp}]
Let $f$ be a non-trivial solution of \eqref{eq:de2}, and let $0\leq r_0<1$ be fixed.
If $r_0<r<1$ and $e^{i\theta}\in\partial\D$, then
\begin{equation*}
  f(re^{i\theta}) = f(r_0 e^{i\theta}) + f'(r_0e^{i\theta}) (re^{i\theta}-r_0 e^{i\theta}) 
  - \int_{r_0 e^{i\theta}}^{re^{i\theta}} f(\zeta)  A(\zeta)  (re^{i\theta}-\zeta) \, d\zeta,
\end{equation*}
by the representation theorem \cite[Theorem~4.1]{H:2000}. Therefore
\begin{equation*}
  |f(r e^{i\theta})| \leq \Big( M_\infty(r_0,f) + M_\infty(r_0,f')(1-r_0) \Big) 
  \exp\!\left(\, \int_{r_0}^r |A(t e^{i\theta})| (1-t) \, dt \right)
\end{equation*}
by Gronwall's lemma \cite[Lemma~5.10]{L:1993}.
This growth estimate, the assumption \eqref{eq:impass},
and the identity $f''=-Af$  imply that $f''\in \mathcal{L}^0$. This completes the proof as
$f\in\mathcal{B}$ by \cite[Theorem~5.4]{Z:2007}.
\end{proof}

We proceed to show that Theorem~\ref{thm:imp} sharpens \cite[Corollary~4(b)]{HKR:2016}.


\begin{remark} \label{remark:imp}
Suppose that the coefficient condition in \cite[Corollary~4(b)]{HKR:2016} holds,
that is, $A\in\mathcal{H}(\D)$ and 
\begin{equation} \label{eq:bz}
  \sup_{z\in\D} \, |A(z)| (1-|z|)^2 \int_0^{|z|} \frac{dr}{1-r} < 1.
\end{equation}
Fix any $0<r_0<1$, and compute
\begin{align*}
  & \sup_{r_0<r<1} \, M_\infty(r,A) (1-r)^2  \exp \!\left( \, \int_{r_0}^r M_\infty(t,A) (1-t) \, dt \right)\\
  & \qquad \leq \sup_{r_0<r<1} \,\frac{1}{\log\frac{1}{1-r}} \, \exp\!\left( \log\log\frac{1}{1-r} - \log\log\frac{1}{1-r_0} \right)
    = \frac{1}{\log\frac{1}{1-r_0}}.
\end{align*}
Therefore, the assumptions of Theorem~\ref{thm:imp} are satisfied.
We point out that Theorem~\ref{thm:imp} applies also to
cases such as $A(z) = (1-z)^{-2} ( \log (e/(1-z)))^{-1}$, $z\in\D$, for which \cite[Corollary~4(b)]{HKR:2016} 
is inconclusive. In particular, Theorem~\ref{thm:imp} can be utilized even with
equality in \eqref{eq:bz}. 

Under an~additional smoothness assumption,
the coefficient condition~\eqref{eq:impass} 
falls also into the category $A\in\mathcal{L}^1$.
This is the case, for example, if $M_\infty(r,A)(1-r)^2(\log (e/(1-r))$, $r_0<r<1$,
is increasing. \hfill $\diamond$
\end{remark}

The following example shows that, even if $\nm{A}_{\mathcal{L}^0}$ is arbitrarily small, it
is possible that \emph{all} non-trivial solutions of \eqref{eq:de2} lie outside $\mathcal{B}$.


\begin{example} \label{ex:nonbloch}
Let $1<\gamma<\infty$ be fixed. The differential equation \eqref{eq:de2} for 
$A(z)=(1-\gamma^2)/(1-z^2)^2$, $z\in\D$, admits linearly independent solutions
\begin{equation*}
  f_1(z) = \frac{(1+z)^{(\gamma+1)/2}}{(1-z)^{(\gamma-1)/2}}, \quad
  f_2(z) = \frac{(1-z)^{(\gamma+1)/2}}{(1+z)^{(\gamma-1)/2}}, \quad z\in\D,
\end{equation*}
which clearly satisfy $f_1,f_2\notin \mathcal{B}$. Since the singularities of $f_1,f_2$
are located at distinct points, we conclude that all linear combinations of $f_1,f_2$,
and therefore all non-trivial solutions of \eqref{eq:de2}, lie outside $\mathcal{B}$.
 \hfill $\diamond$
\end{example}


\section{Proof of Theorem~\ref{thm:nn}}

We begin with an~auxiliary result, which
shows that the coefficient condition $A\in\mathcal{L}^1$ is 
associated with solutions of at most logarithmic growth.
This should be compared to the case of the coefficient condition $A\in\mathcal{L}^0$,
which implies that all solutions of \eqref{eq:de2} satisfy $\sup_{z\in\D} |f(z)| (1-|z|^2)^p < \infty$
for sufficiently large $p=p(\nm{A}_{\mathcal{L}^0})< \infty$, see \cite[Example~1]{P:1982}.


\begin{lemma} \label{lemma:reprelog}
Let $A\in \mathcal{L}^1$.
\begin{enumerate}
\item[\rm (i)]
  All solutions $f$ of \eqref{eq:de2} satisfy
  \begin{equation} \label{eq:logest}
    \sup_{z\in\D} \, |f(z)| \left( \log\frac{e}{1-|z|} \right)^{-\alpha} < \infty
  \end{equation}
  for $\nm{A}_{\mathcal{L}^1}/4 < \alpha<\infty$.

\item[\rm (ii)]
  Any solution $f$ of \eqref{eq:de2}, which satisfies \eqref{eq:logest} for $\alpha=1$,
  belongs to $\mathcal{B}$.
\end{enumerate}
\end{lemma}

The proof of Lemma~\ref{lemma:reprelog}(i) resembles that of \cite[Theorem~2]{GR:2017};
a~similar estimate could also be obtained from \cite[Theorem~4.2]{H:2000}.
Lemma~\ref{lemma:reprelog}(ii) is an~immediate consequence of \eqref{eq:bb} and \cite[Theorem~5.4]{Z:2007},
but plays an~important role in the proof of Theorem~\ref{thm:nn}.


\begin{proof}[Proof of Lemma~\ref{lemma:reprelog}]
(i) Let $f$ be a solution of \eqref{eq:de2}, and $0\leq \delta<R<1$. 
Since
\begin{equation*}
  |f(z)| 
 \leq \int_\delta^{|z|} \!\!\!\int_\delta^t \, \big| f''(sz/|z|) \big| \, ds dt 
  +M_\infty(\delta,f') + M_\infty(\delta,f), \quad \delta<|z|<1,
\end{equation*}
we obtain
\begin{align*}
  \sup_{\delta < |z| <R} \, \frac{|f(z)|}{\big( \log\frac{e}{1-|z|} \big)^\alpha}
  & \leq \left( \sup_{\delta < |\zeta| <R} \, \frac{|f(\zeta)|}{\big( \log\frac{e}{1-|\zeta|} \big)^\alpha} \right) \nm{A}_{\mathcal{L}^1}
    \, \sup_{\delta < |z| <R} I_\alpha(z) \\
  & \qquad + M_\infty(\delta,f') + M_\infty(\delta,f),
\end{align*}
where $I_\alpha(z)$ is as below. Since
\begin{equation*}
  \lim_{|z|\to 1^-} I_\alpha(z)
  =  \lim_{|z|\to 1^-}   \left( \log\frac{e}{1-|z|} \right)^{-\alpha}\int_0^{|z|} 
  \!\!\!\int_0^t \, \frac{\big( \log\frac{e}{1-s} \big)^{\alpha-1}}{(1-s^2)^2} \,  ds dt = \frac{1}{4\alpha}
\end{equation*}
by the Bernoulli-l'H\^{o}pital theorem, we deduce
\eqref{eq:logest} for $\nm{A}_{\mathcal{L}^1}/4 < \alpha<\infty$ by 
choosing a sufficiently large $0\leq \delta<1$, reorganizing the terms and finally letting $R\to 1^-$.
\end{proof}


\begin{proof}[Proof of Theorem~\ref{thm:nn}]
If $n=1$, then $f\in\mathcal{B}$ follows directly from Lemma~\ref{lemma:reprelog};
first, apply part (i) and then (ii). If $n\geq 2$, then we may assume that $f\in\mathcal{B}$
by the first part of the proof.
Since $\nm{A}_{\mathcal{L}^1} < 2$ by the assumption, every solution $f$ of \eqref{eq:de2} satisfies
\eqref{eq:logest} for $\alpha=1/2$
by Lemma~\ref{lemma:reprelog}(i). Note that $(f^2)''=2(f')^2 - 2 f^2 A$ by \eqref{eq:de2}.
We deduce $(f^2)''\in \mathcal{L}^0$, which implies $f^2\in\mathcal{B}$. 

We proceed by induction. Assume that $f^{k-1}\in \mathcal{B}$ for $2< k\leq n$. 
As above, we know that $f\in\mathcal{B}$. Since $\nm{A}_{\mathcal{L}^1} < 4/n\leq 4/k$
by the assumption, every solution $f$ of \eqref{eq:de2} satisfy \eqref{eq:logest} for 
$\alpha=1/k$ by Lemma~\ref{lemma:reprelog}(i). Now
\begin{equation*}
  (f^k)''=  k f' (f^{k-1})' - k f^k A 
\end{equation*}
by \eqref{eq:de2}. We deduce $(f^k)'' \in \mathcal{L}^0$, which gives $f^k\in\mathcal{B}$. The claim follows.
\end{proof}


\section{Proof of Theorem~\ref{thm:ex}}

The following proof takes advantage of universal covering maps to create a~Bloch
function with special properties. Similar arguments appear in the literature
several times. The idea for the following Bloch construction is borrowed from \cite[p.~229]{CCS:1980}.


\begin{proof}[Proof of Theorem~\ref{thm:ex}]
Let $0<C<\infty$.
By the proof of \cite[Theorem~4]{P:1982}, when applied to $Q(r)=C/(1-r)^2$, there exists 
$g\in \mathcal{B} \setminus \mathcal{N}$
with $\nm{g}_{\mathcal{B}} \lesssim C$ such that $f=e^g \not\in\mathcal{N}$ 
is a~solution of \eqref{eq:de2} for $A=-g''-(g')^2$ with $\nm{A}_{\mathcal{L}^0} \leq 4C$.

Let $\mathcal{Z}=\{ x+iy\in\C : x,y\in\Z\}$ be the set of integral lattice points, 
and let $E$ be its preimage $E = \{z\in\D: f(z) \in \mathcal{Z}\}$. Since $E\subset\D$ is a~countable closed set, $E$ 
has capacity zero and therefore the universal covering map from~$\D$ onto $\D \setminus E$ is an~inner function
\cite{F:1935}; see also \cite[p.~261]{S:1979}. Let this inner function be denoted by $I$. 
The function $f\circ I$ belongs to $\mathcal{B}$ since its image, contained in $\C\setminus \mathcal{Z}$, 
does not contain (schlicht) discs of arbitrarily large radius; see \cite[Theorem~2.6]{C:1979}, for example. 
Note that $f\circ I$ is non-vanishing, and define $B\in\mathcal{H}(\D)$ by 
\begin{equation*}
B = -\frac{(f \circ I)''}{f\circ I} = (A \circ I) (I')^2 - (g' \circ I) \, I''.
\end{equation*}
By the Schwarz-Pick lemma, and its extension \cite[Theorem~2]{R:1985}, we deduce
\begin{equation*}
\begin{split}
  \nm{B}_{\mathcal{L}^0} & \leq \sup_{z\in\D} \, (1-|z|^2)^2 \, 
  \frac{\nm{A}_{\mathcal{L}^0}}{(1-|I(z)|^2)^2} \cdot \frac{(1-|I(z)|^2)^2}{(1-|z|^2)^2} \\
  & \qquad + \sup_{z\in\D} \, (1-|z|^2)^2 \, \frac{\nm{g}_{\mathcal{B}}}{1-|I(z)|^2} \cdot \frac{2! \, (1-|I(z)|^2)}{(1-|z|)^2(1+|z|)}\\
  & \leq \nm{A}_{\mathcal{L}^0} + 4 \, \nm{g}_{\mathcal{B}}.
\end{split}
\end{equation*}
We conclude that $h = f\circ I \in \mathcal{B}$ is a zero-free solution of $h''+Bh=0$,
where $\nm{B}_{\mathcal{L}^0} \lesssim C$
with a comparison constant independent of $C$. Finally, \cite[Proposition~3.3]{S:1979} implies that 
$f \circ I$ does not belong to $\mathcal{N}$.
\end{proof}


\section{Proof of Theorem~\ref{thm:zfbase}}

The following result shows that slow growth of the coefficient
ensures the existence of zero-free solution bases.


\begin{lemma} \label{lemma:zfbase}
If $\nm{A}_{\mathcal{L}^0}\leq 1$, then \eqref{eq:de2} admits linearly independent 
zero-free solutions $f_1$ and $f_2$ such that $\log f_1 - \log f_2 \in\BMOA$.
\end{lemma}

If $A\in\mathcal{L}^0$, then any zero-free solution $f$ of \eqref{eq:de2}
satisfies $\log f\in\mathcal{B}$ by \cite[Theorem~4(ii)]{GNR:preprint}. The contribution
of Lemma~\ref{lemma:zfbase} lies in the fact that linearly independent zero-free solutions
are shown to be closely related to each other.
If $A\in\mathcal{L}^1$, then any zero-free solution 
$f$ of \eqref{eq:de2} satisfies $\log f\in\BMOA$ by \cite[Theorem~4(i)]{GNR:preprint},
and therefore the $\mathcal{L}^1$-counterpart of Lemma~\ref{lemma:zfbase} is trivial.


\begin{proof}[Proof of Lemma~\ref{lemma:zfbase}]
Let $g_1$ and $g_2$ be linearly independent solutions of \eqref{eq:de2}
where $\nm{A}_{\mathcal{L}^0}\leq 1$. It follows that $h=g_1/g_2$ is a~locally
univalent meromorphic (not necessarily analytic) function whose Schwarzian derivative $S_h = 2A$ satisfies 
$\nm{S_h}_{\mathcal{L}^0}\leq 2$, and therefore $h$ is univalent in $\D$ by \cite[Theorem~I]{N:1949}.
Consequently, there exist two distinct values $\zeta_1,\zeta_2\in\C \cup \{\infty\}$
which belong to the complement of $h(\D)$ with respect to the extended complex plane.
If $\zeta_j\in\C$ then define $f_j = g_1 - \zeta_j g_2$, while otherwise let $f_j=g_2$.
We conclude that $f_1$ and $f_2$ are linearly independent zero-free
solutions of \eqref{eq:de2}.

Finally, $w=f_1/f_2$ is a~locally univalent analytic zero-free function, 
whose Schwarzian derivative agrees with $S_h$. It follows that $w$ is univalent,
and therefore $\log w \in\BMOA$ by \cite[Corollary~1, p.~21]{B:1980}. The claim follows.
\end{proof}


\begin{proof}[Proof of Theorem~\ref{thm:zfbase}]
Let $f_1$ and $f_2$ be linearly independent non-vanishing solutions of \eqref{eq:de2}.
Their existence follows from $\nm{A}_{\mathcal{L}^0}\leq 1$ as in the
proof of Lemma~\ref{lemma:zfbase}. Without loss of generality, we may assume 
that $f_2\in\bigcup_{0<p<\infty} H^p$ is the zero-free solution given by the hypothesis.
Any solution $f$ of \eqref{eq:de2} can be represented in the form
$f=\alpha f_1 + \beta f_2= f_2 \, ( \alpha \, e^{\log f_1- \log f_2} + \beta ),$
where $\alpha,\beta\in\C$ are constants depending on $f$. Since
$\log f_1 - \log f_2\in\BMOA$ by Lemma~\ref{lemma:zfbase}, 
we deduce $\exp(\log f_1 - \log f_2)\in \bigcup_{0<p<\infty} H^p$ by \cite[Theorem~1]{CS:1976}.
This proves the assertion.
\end{proof}


\section{Proof of Theorem~\ref{thm:bmoa_not}}

Theorem~\ref{thm:bmoa_not} reveals that finitely valent solutions possess a~unique property,
which is not even found from all bounded analytic functions.
To construct a~bounded function $f\in\mathcal{H}(\D)$ for which \eqref{eq:newest} fails,
consider a~Blaschke sequence which is not a~zero-sequence 
for the weighted Dirichlet space $\mathcal{D}_s$ for fixed $0<s<1$,
and let $f$ be the corresponding Blaschke product.
See \cite[p.~1981]{PP:2011} for more details.


\begin{proof}[Proof of Theorem~\ref{thm:bmoa_not}]
Let $\beta$ be  any constant such that $\nm{A}_{\mathcal{L}^1}/2<\beta<\infty$,
and fix $\alpha$ such that $\nm{A}_{\mathcal{L}^1}/4 < \alpha<\beta/2$.
Since $f$ is a solution of \eqref{eq:de2} for $A\in\mathcal{L}^1$, \eqref{eq:logest} 
holds by Lemma~\ref{lemma:reprelog}(i). As in \cite[p.~593]{P:1977}, we compute
\begin{align*}
  \int_{D(0,r)} |f'(z)|^2 \, dm(z) & = \int_{f(D(0,r))} \Bigg( \sum_{z\in\D\, : f(z)=\zeta} 1 \Bigg) \, dm(\zeta)\\
                                   &  \leq \int_{D(0,M_\infty(r,f))} n(f,\zeta)\,  dm(\zeta)
                                     \leq 4 \big( M_\infty(r,f) + 1 \big)^2 \cdot V_f\\
                                   & \lesssim \left( \log\frac{e}{1-r} \right)^{2\alpha}, \quad 0<r<1,
\end{align*}
by \eqref{eq:intnumber} and the generic change of variable formula \cite[Proposition~2.1]{A:1992}. Now 
\begin{align*}
  & \int_{\D} |f'(z)|^2 \left( \log\frac{e}{1-|z|} \right)^{-\beta} \, dm(z) \\
  & \qquad = \int_0^1 \left( \int_{D(0,r)} |f'(z)|^2 \, dm(z) \right) \frac{\beta}{(1-r) \big( \log\frac{e}{1-r} \big)^{\beta+1}} \, dr \\
  & \qquad \lesssim \int_0^1 \frac{dr}{(1-r) \big( \log\frac{e}{1-r} \big)^{1+ \beta-2\alpha}}  < \infty
\end{align*}
by Fubini's theorem. 
\end{proof}

The following example concerns the valence of solutions of \eqref{eq:de2}.


\begin{example} \label{ex:valent}
Let $0<\alpha<1$.
As in \cite[Example~5(b)]{HKR:2016}, we conclude that $f(z)=(\log(e/(1-z)))^\alpha$
is a solution of \eqref{eq:de2} for
\begin{equation*}
  A(z)=\frac{-\alpha}{(1-z)^2} \bigg( (\alpha-1)\left(\log\frac{e}{1-z} \right)^{-2}
  +\left(\log\frac{e}{1-z}\right)^{-1} \bigg), \quad z\in\D,
\end{equation*}
where $\nm{A}_{\mathcal{L}^1}\lesssim \alpha$.
Since $z\mapsto \log(e/(1-z))$ is univalent in $\D$, 
we see that $f$ is finitely valent for $\alpha\in (0,1) \cap \mathbb{Q}$ and
infinitely valent for $\alpha\in (0,1) \setminus \mathbb{Q}$. \hfill $\diamond$
\end{example}


\section{Proofs of Theorems~\ref{thm:conv_preli} and~\ref{thm:conv}}

Let $f_1$ and $f_2$ be linearly independent solutions of \eqref{eq:de2} for $A\in\mathcal{H}(\D)$. 
We may assume that the Wronskian determinant satisfies $f_1f_2'-f_1'f_2 = 1$. Differentiate this identity
once to obtain $f_1f_2''-f_1''f_2 = 0$, and differentiate it twice to deduce
$f_1'''f_2 - f_1 f_2''' = f_1'f_2'' - f_1''f_2' = A$, where the last equality follows from \eqref{eq:de2}.


\begin{proof}[Proof of Theorem~\ref{thm:conv_preli}]
Since $f_1,f_2\in \mathcal{B}$, we conclude that $f_1'',f_2''\in \mathcal{L}^0$. 
Define $h(z) = |f_1(z)| + |f_2(z)|$ for $z\in\D$.
Function $h$ is non-vanishing as the Wronskian determinant satisfies $f_1f_2'-f_1'f_2 = 1$. On one hand, 
\begin{equation*}
\begin{split}
  |A(z)| & = \frac{|f_1(z) A(z)| + |f_2(z) A(z)|}{|f_1(z)| + |f_2(z)|} = \frac{|f_1''(z)|+|f_2''(z)|}{|f_1(z)| + |f_2(z)|}\\
  & \lesssim \frac{\max\{\nm{f_1}_{\mathcal{B}},\nm{f_2}_{\mathcal{B}}\}}{(1-|z|^2)^3} \cdot \frac{1-|z|^2}{h(z)}, \quad z\in\D,
\end{split}
\end{equation*}
with an~absolute comparison constant. On the other hand, 
\begin{equation*}
  |A(z)| \leq |f_1'''(z)| |f_2(z)| + |f_1(z)| |f_2'''(z)|
  \lesssim \frac{\max\{\nm{f_1}_{\mathcal{B}},\nm{f_2}_{\mathcal{B}}\}}{(1-|z|^2)^3} \, h(z), \quad z\in\D.
\end{equation*}

Since $\min\{ x/y, y \} \leq \sqrt{x}$ for all $0<x,y<\infty$, we obtain
\begin{equation*}
\min\left\{ (1-|z|^2)/h(z), \, h(z)\right\} \leq \sqrt{1-|z|^2}, \quad z\in\D.
\end{equation*}
The assertion $\sup_{z\in\D} |A(z)|(1-|z|^2)^{5/2}\lesssim\max\{\nm{f_1}_{\mathcal{B}},\nm{f_2}_{\mathcal{B}}\}$ follows. 
\end{proof}

The proof of Theorem~\ref{thm:conv} is similar to the one above, with the difference
that the auxiliary function $h$ in the proof of Theorem~\ref{thm:conv_preli},
is now uniformly bounded away from zero by the assumption.


\begin{proof}[Proof of Theorem~\ref{thm:conv}]
Since $f_1,f_2\in \mathcal{B}$, we have $f_1'',f_2''\in \mathcal{L}^0$. By \eqref{eq:de2},
\begin{equation*}
\begin{split}
  \sup_{z\in\D} \, |A(z)| (1-|z|^2)^2 & = \sup_{z\in\D}\, \frac{|f_1(z) A(z)| + |f_2(z) A(z)|}{|f_1(z)| + |f_2(z)|} \, (1-|z|^2)^2 \\
  & \leq \left( \, \inf_{z\in\D} \big( |f_1(z)| + |f_2(z)| \big) \right)^{-1} \big( \nm{f_1''}_{\mathcal{L}^0} + \nm{f_2''}_{\mathcal{L}^0} \big).
\end{split}
\end{equation*}
Let $w=f_1/f_2$, which implies that $w'=-1/f_2^2$.
To see that $w^{\#}$ is bounded in $\D$, it suffices to write
\begin{equation} \label{eq:mmest}
\begin{split}
  \sup_{z\in\D} \, w^{\#}(z) & = \sup_{z\in\D} \, \frac{1}{|f_1(z)|^2 + |f_2(z)|^2} 
  \leq  \sup_{z\in\D} \, \frac{2}{\big( |f_1(z)| + |f_2(z)| \big)^2}  \\
  & \leq 2  \left( \, \inf_{z\in\D} \big( |f_1(z)| + |f_2(z)| \big) \right)^{-2}.
\end{split}
\end{equation}
This completes the proof.
\end{proof}

We take the opportunity to mention an~interesting application of \eqref{eq:mmest}.
Let $f_1,f_2$ be linearly independent solutions of \eqref{eq:de2} for $A\in\mathcal{H}(\D)$ 
such that $\inf_{z\in\D} \, (|f_1(z)|+ |f_2(z)|) >0$,
and let $z_1,z_2\in\D$ be (necessarily distinct) points at which $f_1(z_1)=0=f_2(z_2)$. 
Let $\gamma(z_1,z_2)$ denote the straight line segment from $z_1$ to~$z_2$.
Since $z_1$ is a~(simple) zero of~$w=f_1/f_2$, and $z_2$ is a~(simple) pole of~$w$,
we deduce
\begin{equation*}
  1 \lesssim \int_{\gamma(z_1,z_2)} \frac{|w'(z)|}{1+|w(z)|^2} \, |dz| 
  \leq \left( \, \sup_{z\in\D} \, w^{\#}(z) \right) | z_1 - z_2 |
\end{equation*}
as the spherical length of $w(\gamma(z_1,z_2))$ is uniformly bounded from below.
Therefore, \eqref{eq:mmest} implies that $|z_1-z_2|$ is uniformly bounded away from zero.


\section{Proof of Theorem~\ref{thm:bmoa}}

By \cite[Theorem~1]{Z:2003} and the subharmonicity of $|A|$, we deduce
\begin{equation*}
  \nm{A}_{\mathcal{L}^1} \lesssim \, \sup_{a\in\D}  \left( \log\frac{e}{1-|a|} \right) \int_{\D} |A(z)| (1-|\varphi_a(z)|^2) \, dm(z).
\end{equation*}
Consequently, when proving Theorem~\ref{thm:bmoa}, we may assume that all solutions
of \eqref{eq:de2} are in $\mathcal{B}$ by \cite[Corollary~4(b)]{HKR:2016} or Theorem~\ref{thm:imp}.


\begin{proof}[Proof of Theorem~\ref{thm:bmoa}]
Let $f$ be a solution of \eqref{eq:de2} and consider its normalized hyperbolic translates 
$g_a(z) = f(\varphi_a(z)) - f(a)$ for $a\in\D$. To prove $f\in\BMOA$ it suffices to show that
\begin{equation} \label{eq:need}
  \sup_{a\in\D} \sup_{0<r<1} m(r,g_a) 
  = \sup_{a\in\D} \sup_{0<r<1} \, \frac{1}{2\pi} \int_0^{2\pi} \log^+ |g_a(re^{i\theta})| \, d\theta 
  < \infty
\end{equation}
by \cite[Corollary~2, p.~15]{B:1980}. We proceed to verify that the proximity functions $m(r,g_a)$
satisfy \eqref{eq:need}.

A straight-forward computation reveals that $g_a\in\mathcal{H}(\D)$ is a solution of 
the non-homogenous linear differential equation
\begin{equation*}
  g_a'' + B_a \, g_a' + C_a \, g_a = - f(a) \, C_a,
\end{equation*}
where $B_a,C_a\in\mathcal{H}(\D)$ are given by
\begin{equation*}
  B_a(z) = - \, \frac{\varphi_a''(z)}{\varphi_a'(z)}, 
  \quad C_a(z) =A\big(\varphi_a(z)\big) \, \varphi_a'(z)^2, 
  \quad z\in\D.
\end{equation*}
By \cite[Corollary~3(a)]{HKR:2009}, we deduce
\begin{align*}
  m(r,g_a) & \lesssim 1 + \log^+ \!\big( |f'(a)|(1-|a|^2) \big) \\
  & \qquad + \int_0^{2\pi} \log^+ \!\left( \int_0^r |A(\varphi_a(se^{i\theta}))| |\varphi_a'(se^{i\theta})|^2 |f(a)| (1-s) \, ds \right)  d\theta\\
  & \qquad + \int_{D(0,r)} |A(\varphi_a(z))| |\varphi_a'(z)|^2 (1-|z|^2) \, dm(z)\\
  & \qquad + \int_{D(0,r)} \left| \frac{\varphi_a''(z)}{\varphi_a'(z)} \right|  dm(z)
    + \int_{D(0,r)} \left| \left( \frac{\varphi_a''}{\varphi_a'} \right)'\!(z)\right| (1-|z|^2) \, dm(z)
\end{align*}
for all $0<r<1$, where the comparison constant is independent of $a\in\D$. The area integrals
involving $B_a$ and $B_a'$ are uniformly bounded for $0<r<1$ and $a\in\D$ by standard estimates, while
\begin{equation*}
  \sup_{a\in\D} \sup_{0<r<1} \, \int_{D(0,r)} |A(\varphi_a(z))| |\varphi_a'(z)|^2 (1-|z|^2) \, dm(z)
\end{equation*}
is at most \eqref{eq:bmoaa} by a conformal change of variable. Recalling that 
$\log^+ x \leq x$ for all positive $x$, we conclude
\begin{align}
  & \sup_{a\in\D} \sup_{0<r<1} \, \int_0^{2\pi} \log^+ 
    \left( \int_0^r |A(\varphi_a(se^{i\theta}))| |\varphi_a'(se^{i\theta})|^2 |f(a)| (1-s) \, ds \right)  d\theta \notag\\
  & \qquad \lesssim \sup_{a\in\D} \, M_\infty(|a|, f)  \int_{\D} |A(z))| (1-|\varphi_a(z)|^2) \, dm(z) \label{eq:finn}.
\end{align}
The quantity \eqref{eq:finn} is finite by \eqref{eq:bmoaa} and the fact $f\in\mathcal{B}$.
This proves \eqref{eq:need}, and hence Theorem~\ref{thm:bmoa}.
\end{proof}


\section{Proofs of Theorem~\ref{thm:vmoa} and Corollary~\ref{cor:vmoa}} \label{sec:slast}

The proof of Theorem~\ref{thm:vmoa} is based on the following result \cite[Corollary~5.3]{R:2003}:
If $f\in\mathcal{H}(\D)$ and 
\begin{equation} \label{eq:vmoan}
\int_0^1 M_\infty(r,f'')^2 (1-r^2)^3 \, dr<\infty,
\end{equation}
then $f\in\VMOA$.


\begin{proof}[Proof of Theorem~\ref{thm:vmoa}]
Let $f$ be a non-trivial solution of \eqref{eq:de2}, and let $0<r_0<1$ be fixed.
As in the proof of Theorem~\ref{thm:imp}, we obtain
\begin{align*}
M_\infty(r,f'')  \leq M_\infty(r,A) \, M_\infty(r,f)
   \lesssim  M_\infty(r,A)  \exp\!\left(\, \int_{r_0}^r M_\infty(t,A) (1-t) \, dt \right)
\end{align*}
for $r_0<r<1$, where the comparison constant  is independent of $r$. 
The assertion follows as $f\in\VMOA$ by \eqref{eq:vmoageneral} and \eqref{eq:vmoan}.
\end{proof}


\begin{proof}[Proof of Corollary~\ref{cor:vmoa}]
Fix any $0<r_0<1$. The coefficient condition \eqref{eq:corvmoa} implies
that there exists an~absolute constant $0<C<\infty$ such that
\begin{equation*}
 \exp\!\left(\, 2 \int_{r_0}^r M_\infty(t,A) (1-t) \, dt \right)
 \lesssim \left( \log\log\frac{e}{1-r} \right)^{2C}, \quad r_0<r<1,
\end{equation*}
where the comparison constant is independent of $r$. The condition
\eqref{eq:vmoageneral} is satisfied by a~straight-forward computation, which concludes the proof.
\end{proof}

The following example shows that the coefficient condition~\eqref{eq:corvmoa} allows
solutions of \eqref{eq:de2} to be unbounded.


\begin{example} \label{ex:unbounded}
Let $0<\alpha<\infty$. Note that
$f(z) = (\log\log e^e/(1-z))^\alpha$, $z\in\D$, is a~zero-free unbounded solution of \eqref{eq:de2} for
\begin{equation*}
  A(z) = - \alpha \, \frac{\alpha - 1 + \left( \log\frac{e^e}{1-z} - 1\right) \!\left( \log \log\frac{e^e}{1-z}\right)}
  {(1-z)^2\left( \log\frac{e^e}{1-z} \right)^2\left( \log \log\frac{e^e}{1-z}\right)^2}, \quad z\in\D.
\end{equation*}
It is immediate that \eqref{eq:corvmoa} is satisfied.
\hfill $\diamond$
\end{example}


\end{document}